\documentclass[11pt,a4paper,reqno]{amsart}

\usepackage[latin9]{inputenc}
\usepackage[includehead,includefoot,margin=25mm]{geometry}
\usepackage{times}
\usepackage{amstext}
\usepackage{amsthm}
\usepackage{amssymb}
\usepackage{enumerate}
\usepackage{color}
\usepackage{xypic}

\makeatletter
\numberwithin{equation}{section}
\numberwithin{figure}{section}
\theoremstyle{plain}
\newtheorem{thm}{\protect\theoremname}[section]
\newtheorem*{thm*}{\protect\theoremname}
  \theoremstyle{plain}
  \newtheorem{lem}[thm]{\protect\lemmaname}
  \theoremstyle{definition}
  \newtheorem{defn}[thm]{\protect\definitionname}
  \theoremstyle{definition}
  \newtheorem{notation}[thm]{\protect\notationname}
  \theoremstyle{plain}
  \newtheorem{prop}[thm]{\protect\propositionname}
  \theoremstyle{plain}
  \newtheorem{cor}[thm]{\protect\corollaryname}
  \theoremstyle{remark}
  \newtheorem{rem}[thm]{\protect\remarkname}
  \theoremstyle{definition}
  \newtheorem{example}[thm]{\protect\examplename}
 \theoremstyle{definition}
  \newtheorem{conv}[thm]{\protect\conventionname}

\makeatother

\usepackage[english]{babel}
  \providecommand{\corollaryname}{Corollary}
  \providecommand{\definitionname}{Definition}
  \providecommand{\examplename}{Example}
  \providecommand{\lemmaname}{Lemma}
  \providecommand{\notationname}{Notation}
  \providecommand{\propositionname}{Proposition}
  \providecommand{\remarkname}{Remark}
\providecommand{\theoremname}{Theorem}
\providecommand{\conventionname}{Convention}

\newcommand{\CC}[0]{\ensuremath{\mathbb{C}}}
\newcommand{\ZZ}[0]{\ensuremath{\mathbb{Z}}}

\newcommand{\GM}[0]{\ensuremath{\mathbb{G}_{m,\mathbb{C}}}}

\newcommand{\RR}[0]{\ensuremath{\mathbb{R}}}
\newcommand{\QQ}[0]{\ensuremath{\mathbb{Q}}}

\newcommand{\KK}[0]{\ensuremath{\mathrm{k}}}
\newcommand{\OO}[0]{\ensuremath{\mathcal{O}}}
\newcommand{\DD}[0]{\ensuremath{\mathcal{D}}}

\newcommand{\spec}[0]{\ensuremath{\operatorname{Spec}}}
\newcommand{\Aut}[0]{\ensuremath{\operatorname{Aut}}}
\newcommand{\divi}[0]{\ensuremath{\operatorname{div}}}

\newcommand {\flip}[1]{\widehat{#1}}

\begin{document}
\author{Adrien Dubouloz} \address{IMB UMR5584, CNRS, Univ. Bourgogne
  Franche-Comt\'e, F-21000 Dijon, France.}
\email{adrien.dubouloz@u-bourgogne.fr}

\author{Alvaro Liendo} \address{Instituto de Matem\'atica y F\'\i
  sica, Universidad de Talca, Casilla 721, Talca, Chile}
\email{aliendo@inst-mat.utalca.cl}

\thanks{{\it 2010 Mathematics Subject
    Classification}:  14P05, 14L30.\\
  \mbox{\hspace{11pt}}{\it Key words}: Circle actions, torus actions,
  real varieties.\\
  \mbox{\hspace{11pt}}The first author was partially supported by ANR
  Project FIBALGA ANR-18-CE40-0003-01 and the second author was
  partially supported by Fondecyt project 11121151 and by the grant
  346300 for IMPAN from the Simons Foundation and the matching
  2015-2019 Polish MNiSW fund}

\title{Normal real affine varieties with circle actions}
\begin{abstract}
  We provide a complete description of normal affine algebraic
  varieties over the real numbers endowed with an effective action of
  the real circle, that is, the real form of the complex
  multiplicative group whose real locus consists of the unitary circle
  in the real plane.  Our approach builds on the
  geometrico-combinatorial description of normal affine varieties with
  effective actions of split tori in terms of proper polyhedral
  divisors on semiprojective varieties due to Altmann and Hausen.
\end{abstract}

 \dedicatory{Dedicated to Mikhail Zaidenberg on his 70th birthday}
\maketitle

\section*{Introduction}

Normal algebraic varieties $X$ over a field $k$ endowed with actions
of split tori $\mathbb{T}=\mathbb{G}_{m,k}^n$ are quite well
understood in terms of various geometrico-combinatorial
presentations. The case where $\mathbb{T}$ acts faithfully on $X$ and
$\dim(\mathbb{T})=\dim(X)$ is known as toric variety and was first
studied by Demazure in \cite{De70}. These varieties are fully
described in combinatorial terms by means of suitable collections of
convex polyhedral cones in the real vector space $N_\RR=N\otimes_\ZZ \RR$ obtained from the lattice $N$ of $1$-parameter
subgroups of $\mathbb{T}$.  Successive further generalizations
\cite{Dem88,FZ03,AH,AHS08} have led to complete descriptions of normal
$k$-varieties endowed with $\mathbb{T}$-actions in terms of certain
collections of so-called \emph{polyhedral divisors}, which are Weil
divisors $\DD$ on suitable rational quotients for the action, whose
coefficients are convex rational polyhedra in the vector space
$N_\RR$.

For normal algebraic $k$-varieties endowed with actions of non-split
tori, that is, algebraic groups $G$ defined over $k$ whose base
extensions to an algebraic closure $\overline{k}$ of $k$ are
isomorphic to split tori $\mathbb{G}_{m,\overline{k}}^n$ but which are
not isomorphic over $k$ to $\mathbb{G}_{m,k}^n$, much less is known
regarding the existence of geometrico-combinatorial descriptions
similar to the split case.  Toric varieties with respect to non-split
tori have been considered by several authors, see for instance
\cite{Vo82,EL14,Du16}. In another direction, the geometrico-combinatorial
presentation of Altmann-Hausen was partially extended by Langlois \cite{La15}
to yield a description of affine varieties $X$ endowed with 
an effective action of a non-split torus $G$ of dimension $\dim(X)-1$.
Nevertheless, the general case remains elusive. 
A natural and crucial step towards a geometrico-combinatorial description of such varieties 
would be to extend the Altmann-Hausen presentation in terms of polyhedral divisors
\cite{AH} to arbitrary normal affine $k$-varieties $X$ endowed with
effective actions of tori $G$, split or not. Since every torus $G$
splits after base change to a finite Galois extension $K/k$ of $k$,
this naturally leads to seek for such an extension in the form of a
geometrico-combinatorial description of affine $K$-varieties $X$ with
effective actions of split tori $\mathbb{T}$ which are compatible with
additional Galois descent data on $X$ and $\mathbb{T}$ for the finite
Galois cover $\mathrm{Spec}(K)\rightarrow \mathrm{Spec}(k)$.

In this article we lend support to this approach by considering a
simple case of independant geometric interest for which both the
combinatorics and the Galois descent machinery are reduced to their
minimum: normal real affine varieties with an effective action of the
unit circle \[\mathbb{S}^1=\spec(\mathbb{R}[x,y]/(x^2+y^2-1)),\] the
only non-split real form of $\mathbb{G}_{m,\mathbb{R}}$. In this
context, a descent datum on a normal complex affine variety $V$ for
the Galois cover
$\mathrm{Spec}(\mathbb{C})\rightarrow \mathrm{Spec}(\mathbb{R})$ boils
down to anti-regular involution $\sigma$ of $V$ called a \emph{real
  structure}. 

Our main results, Theorem~\ref{thm:real-AH-Main} and
Theorem~\ref{prop:Equiv-Iso}, give a description of $\mathbb{S}^1$-actions 
in the language of \cite{AH} extended to complex affine varieties with real structures. We chose to use the
Altmann-Hausen formalism since it is particularly suited for a
generalization to any non-necessarily split algebraic torus over a
field of characteristic zero which we will tackle in a near
future. Nevertheless, it is well-known that for a $1$-dimensional
split torus, polyhedral divisors in the sense of \cite{AH} correspond
equivalently to data consisting of certain pairs of Weil
$\mathbb{Q}$-divisors, a formalism first used by Dolgachev, Pinkham
and Demazure and then extended by Flenner-Zaidenberg to describe
$\GM$-actions on normal complex affine surfaces via $\QQ$-divisors on their quotients
(see the references in \cite{FZ03}). Corollary~\ref{cor:DPD-General}
provides a description of normal real affine varieties with
an effective $\mathbb{S}^1$-action in this equivalent language, which
can be summarized as follows:

\begin{thm*} A normal real affine variety $X$ endowed with an effective $\mathbb{S}^{1}$-action is uniquely determined by the following data: 
  \begin{enumerate}[(1)]
  \item A normal real semiprojective variety $Z$ corresponding to a
    normal complex semi-projective variety with real structure
    $(Y,\tau)$ representing the "real Altmann-Hausen quotient" of $X$
    by $\mathbb{S}^1$ (see Definition~\ref{def:AH-quotient})

  \item A pair $(D,h)$ consisting of a big and semiample
    $\mathbb{Q}$-Cartier divisor $D$ on $Y$ and $\tau$-invariant
    rational function $h$ on $Y$ satisfying $D+\tau^*D\leq \divi(h)$.

  \end{enumerate}
\end{thm*}

The contents of the article is as follows. In Section~\ref{section1}
we recall the classical equivalence of categories between
quasi-projective real varieties and quasi-projective complex varieties
equipped with a real structure. We establish in
Lemma~\ref{lem:S1-action-equiv-form} the corresponding representation 
of quasi-projective real varieties with circle actions under this equivalence of categories. 
In section~\ref{section2} we establish the main classification results
extending the description by Altmann and Hausen for split torus actions to the case of circle actions on normal real affine varieties. 
Finally, in Sections~\ref{section3} and \ref{section4} we present several instances 
of applications of our techniques to examples taken 
from algebraic and differential geometry.

\section{Basic facts on real algebraic varieties and circle
  actions} \label{section1}

In what follows, we identify the field $\mathbb{R}$ of real numbers
with a subfield of $\mathbb{C}$ via the standard inclusion
$j^*:\mathbb{R}\hookrightarrow\mathbb{C}=\mathbb{R}[i]/(i^{2}+1)$ so
that the usual complex conjugation $J:\mathbb{C}\rightarrow \mathbb{C}$, $z\mapsto\overline{z}$ coincides
with the homomorphism of $\mathbb{R}$-algebra defined by $i\mapsto-i$.

The term $\KK$-variety, where $\KK=\mathbb{R}$ or
$\mathbb{C}$, will refer to a geometrically integral scheme $X$ of
finite type over $\KK$. A morphism of $\KK$-varieties is
a morphism of $\KK$-schemes. 

\subsection{\label{subsec:Real-structures} Real quasi-projective
  varieties as complex varieties with real structures}

Let us briefly recall the classical correspondence \cite{BS64} between
quasi-projective real algebraic varieties and quasi-projective complex
algebraic varieties equipped with a real structure.

Every complex algebraic variety $p:V\rightarrow\mathrm{Spec}(\mathbb{C})$ can be viewed as an
$\mathbb{R}$-scheme $j\circ p:V\rightarrow\mathrm{Spec}(\mathbb{R})$,
and a \emph{real structure} on such a variety $V$ is an involution
$\sigma:V\rightarrow V$ of $\mathbb{R}$-schemes such that
$p\circ\sigma=J\circ p$, where $J$ denotes the complex conjugation. 

Every complex variety $X_{\mathbb{C}}=X\times_{\mathrm{Spec}(\mathbb{R})}\mathrm{\mathrm{Spec}(\mathbb{C})}$
obtained from a real algebraic variety $X$ by the base change $\mathrm{Spec}(\mathbb{C})\rightarrow \mathrm{Spec}(\mathbb{R})$
is canonically endowed with a real structure $\mathrm{\sigma_{X}}=\mathrm{id}_{X}\times J$ for which the morphism
$\mathrm{pr}_{1}:X_{\mathbb{C}}\rightarrow X$ coincides with the
quotient $X_{\mathbb{C}}/\langle\sigma_{X}\rangle$.  Conversely, if
$p:V\rightarrow\mathrm{Spec}(\mathbb{C})$ is equipped with a real
structure $\sigma$ and covered by $\sigma$-invariant affine open
subsets -so for instance if $V$ is quasi-projective-, then the
quotient $\pi:V\rightarrow V/\langle\sigma\rangle$ exists in the
category of schemes and the structure morphism
$p:V\rightarrow\mathrm{Spec}(\mathbb{C})$ descends to a morphism
$V/\langle\sigma\rangle\rightarrow\mathrm{Spec}(\mathbb{R})=\mathrm{Spec}(\mathbb{C})/\langle\tau\rangle$
making $V/\langle\sigma\rangle$ into a real algebraic variety $X$ such
that $V\simeq X_{\mathbb{C}}$.  This correspondence extends to a
well-known equivalence of categories which can be summarized as
follows:
\begin{lem}
\label{lem:Real-structures} The category of quasi-projective real
algebraic varieties is equivalent to the category $\mathcal{C}$ whose
objects are pairs $(V,\sigma)$ consisting of a quasi-projective complex
algebraic variety $V$ and a real structure $\sigma:V\rightarrow V$
and whose morphisms $(V,\sigma)\rightarrow(V',\sigma')$ are morphisms
of complex algebraic varieties $f:V\rightarrow V'$ such that $\sigma'\circ f=f\circ\sigma$. 
\end{lem}

In particular, two real structures $\sigma$ and $\sigma'$ on the same
quasi-projective complex variety $V$ define isomorphic real algebraic
varieties $V/\langle\sigma\rangle$ and $V/\langle\sigma'\rangle$ if
and only if there exists an isomorphism of complex algebraic varieties
$f:V\rightarrow V$ such that $\sigma'\circ f=f\circ\sigma$.

\medskip 

In the sequel, we will indistinctly represent a quasi-projective real
variety $X$ by a pair $(V,\sigma)$ where $V$ is a quasi-projective complex variety and
$\sigma$ is a real structure on $V$ such that $V/\langle\sigma\rangle$
is isomorphic to $X$.  Similarly, we will represent a morphism
(resp. a rational map) $f:X\rightarrow X'$ between real varieties
represented by pairs $(V,\sigma)$ and $(V',\sigma')$ respectively by a
morphism (resp. a rational map) $\tilde{f}:V\rightarrow V'$ such that
$\sigma'\circ \tilde{f}=\tilde{f}\circ \sigma$. We sometimes abbreviate this
condition by saying that $\tilde{f}$ is a \emph{real morphism} (resp
\emph{real rational map}).

\begin{defn}
  A \emph{real form} of a real algebraic variety $X=(V,\sigma)$ is a
  real algebraic variety $X'=(V',\sigma')$ such that $V$ and $V'$ are
  isomorphic as complex varieties. Isomorphy classes of real forms of
  $X$ are classified by the Galois cohomology group
  $H^{1}(\operatorname{Gal}(\CC/\RR),\Aut_\CC(V))$ where the
  non-trivial element of $\operatorname{Gal}(\CC/\RR)=\mu_2$ acts on
  $\Aut_{\CC}(V)$ by conjugation $f\mapsto\sigma f\sigma^{-1}$.
\end{defn}

Recall that an algebraic variety $V$ is said to be
\emph{semi-projective} if its coordinate ring
$\Gamma(V,\mathcal{O}_V)$ is finitely generated and the canonical
morphism $V\rightarrow \mathrm{Spec}(\Gamma(V,\mathcal{O}_V))$ is
projective. When $X$ is a real algebraic variety represented by a pair
$(V,\sigma)$, we denote by $\Gamma(\sigma)$ the unique real structure
on $\mathrm{Spec}(\Gamma(V,\mathcal{O}_V))$ for which the canonical
morphism
$(V,\sigma)\rightarrow
(\mathrm{Spec}(\Gamma(V,\mathcal{O}_V),\Gamma(\sigma))$ is a real
morphism.

\subsection{Circle actions on quasi-projective real
  varieties} \label{subsec:Circle-actions-Defs}
\begin{defn}
  \label{def:Circle}The \emph{real circle} $\mathbb{S}^{1}$ is the
  only non-trivial real form
\[
  R_{\mathbb{C}/\mathbb{R}}^{1}\mathbb{G}_{m,\mathbb{C}}=\mathrm{Spec}(\mathbb{R}[x,y]/(x^{2}+y^{2}-1))
\]
of the multiplicative group $\mathbb{G}_{m,\mathbb{R}}$.  The group
structure on $\mathbb{S}^{1}$ is given by
$$\left(x,y\right)\cdot(x',y')=(xx'-yy',xy'+yx'),$$ and the morphism of
group schemes
\begin{equation}
\begin{array}{ccc}
  \rho_{0}:\mathbb{S}^{1} & \rightarrow & \mathrm{SL}_{2,\mathbb{R}}=\mathrm{Spec}(\mathbb{R}[a_{11},a_{12},a_{21},a_{22}]/(a_{11}a_{22}-a_{12}a_{21}-1))\\
  (x,y) & \mapsto & \left(x,y,-y,x\right)
\end{array}\label{eq:Representation}
\end{equation}
induces an isomorphism between $\mathbb{S}^{1}$ and the closed
subgroup $\mathrm{SO}_{2,\mathbb{R}}$ defined by the equation
$a_{11}-a_{22}=a_{12}+a_{21}=0$.
\end{defn}

It is straighforward to check that the map
\begin{equation} \label{eqn:GroupIso}
  \varphi:\GM=\mathrm{Spec}(\mathbb{C}[t^{\pm1}])\rightarrow
  \mathbb{S}_{\mathbb{C}}^{1},\;t\mapsto(x,y)=((t+t^{-1})/2,(t-t^{-1})/2i)
\end{equation} 
is an isomorphism of complex group schemes. The pull-back of the canonical real
structure $\sigma_{\mathbb{S}^1}$ on $\mathbb{S}_{\mathbb{C}}^{1}$ by
$\varphi$ is the real structure $\rho$ on $\GM$ defined as the
composition of the involution $t\mapsto t^{-1}$, induced by the
involution $-\mathrm{id}_M : m\mapsto -m$ of the character lattice
$M\simeq \ZZ$ of $\GM$, with the complex conjugation. We henceforth
identify the group object $\mathbb{S}^1$ in the category of real
algebraic varieties with the pair $(\GM,\rho)$.

\begin{lem}
  \label{lem:S1-action-equiv-form} 
  There is a one-to-one correspondence between quasi-projective real
  algebraic varieties endowed with an effective
  $\mathbb{S}^{1}$-action and triples $(V,\sigma,\mu)$ consisting of a
  quasi-projective real algebraic variety $X=(V,\sigma)$ and an effective
  $\GM$-action $\mu:\GM\times V\rightarrow V$ such that the following
  diagram commutes:
  \begin{align}
  \label{eq:diag-Gm} 
    \xymatrix{ \GM \times V \ar[rr]^{\mu} \ar[d]_{\rho\times\sigma} & &
    V \ar[d]^{\sigma} \\
    \GM \times V
    \ar[rr]^{\mu} & &  V.}    
  \end{align}
\end{lem}

\begin{proof}
  An $\mathbb{S}^1$-action on $X=(V,\sigma)$ corresponds by definition
  to an effective $\mathbb{S}^1_{\mathbb{C}}$-action $\eta$ on $V$
  such that
  $\sigma\circ\eta=\eta\circ(\sigma_{\mathbb{S}^1}\times \sigma)$,
  hence, composing with the real isomorphism $\varphi$ of
  (\ref{eqn:GroupIso}), to an effective $\GM$-action
  $\mu=\eta\circ(\varphi\times \mathrm{id}_V)$ with the announced
  property.
\end{proof}

\begin{conv}
  In the rest of the article, we will indistinctly represent a
  quasi-projective real variety $X$ endowed with an effective
  $\mathbb{S}^1$-action by one of the following data:
  \begin{enumerate}[(1)]
  \item a quasi-projective real algebraic variety $X$ and a morphism of real algebraic varieties 
    $\mathbb{S}^1\times X\rightarrow X$ defining an effective $\mathbb{S}^1$-action; or

  \item a triple $(V,\sigma,\mu)$ consisting of a quasi-projective
    complex algebraic variety $V$ endowed with a real structure $\sigma$ and a
    morphism of complex algebraic varieties $\mu:\GM \times V\rightarrow V$
    defining an effective $\GM$-action on $V$ such that
    $\mu\circ(\rho\times \sigma)=\sigma\circ\mu$.
  \end{enumerate}
\end{conv}

\begin{defn} A real form of a quasi-projective real
  $\mathbb{S}^1$-variety $X=(V,\sigma,\mu)$ is a quasi-projective real
  $\mathbb{S}^1$-variety $X'=(V',\sigma',\mu')$ such that $V$ is
  $\GM$-equivariantly isomorphic to $V'$.
\end{defn}

\subsection{The case of real affine varieties}
\label{affine-case}
Specializing further to the case where $X=(V,\sigma)$ is a real affine
algebraic variety, say $V=\spec(A)$ for some finitely generated
integral $\mathbb{C}$-algebra $A$, the $\GM$-action $\mu$ on $V$ is
equivalently determined by its co-morphism
$\mu^{*}:A\rightarrow A\otimes_{\mathbb{C}}\mathbb{C}[t^{\pm1}]$.
Recall that a semi-invariant regular function of weight
$m\in\mathbb{Z}$ on $V$ for the action $\mu$ is an element $f\in A$
such that $\mu^{*}f=f\otimes t^{m}$, and that $A$ is then
$\mathbb{Z}$-graded in a natural way by its sub-spaces $A_{m}$ of
semi-invariants of weight $m$ for all $m\in\mathbb{Z}$. The action
$\mu$ is said to be \emph{hyperbolic} if there exists $m<0$ and $m'>0$
such that $A_{m}$ and $A_{m'}$ are non-zero.

\begin{lem} \label{lem:Charac-S1-actions} Let
  $X=(V=\spec(A),\sigma,\mu)$ be a real affine algebraic variety
  endowed with an effective $\mathbb{S}^{1}$-action and let
  $A=\bigoplus_{m\in\mathbb{Z}}A_{m}$ be the decomposition of $A$ into
  semi-invariants sub-spaces for the $\GM$-action $\mu$. Then the
  following hold:
\begin{enumerate}[$(i)$]
\item The action $\mu$ is hyperbolic and $A_m\neq 0$ for all $m\in\mathbb{Z}$
\item For all $m\in\mathbb{Z}$, $\sigma^{*}(A_{m})=A_{-m}$,
\item The restriction of $\sigma^{*}$ to $A_{0}=A^{\GM}$ is the
  co-morphism of a real structure $\overline{\sigma}$ on the algebraic
  quotient $V/\!/\GM=\mathrm{Spec}(A_{0})$ of $V$.
\end{enumerate}
\end{lem}

\begin{proof}
The  commutativity of the diagram \eqref{eq:diag-Gm} implies that for every semi-invariant $f$ of weight $m\in \mathbb{Z}$, we have
 $$\mu^{*}\sigma^{*}(f)=(\sigma^{*}\otimes\rho^{*})(f\otimes  t^{m})=\sigma^{*}(f)\otimes t^{-m},$$ hence that $\sigma^{*}(f)$ is a
 semi-invariant of weight $-m$. The equality
 $\sigma^{*}(A_{m})=A_{-m}$ follows from the fact that $\sigma^{*}$ is
 an automorphism of $A$, which proves the second assertion. Since $\mu$ is non trivial, there exists a semi-invariant function $f$ of non-zero
weight $m$, and hence a semi-invariant function $\sigma^*(f)$ of non-zero weight $-m$. This shows that $\mu$ is hyperbolic, and the second part of the first assertion is a standard fact for such actions. Indeed, since the action is
 effective, the set $\left\{m\in \ZZ\mid A_m\neq \{0\}\right\}$ is not
 contained in any proper sublattice $d\cdot\ZZ$, $d>1$. Hence, there
 exists $e<0$ and $e'>0$ relatively prime such that $A_{e}$ and
 $A_{e'}$ are non-zero. Let $f$ and $g$ be non-zero elements in $A_e$
 and $A_{e'}$, respectively.  Now, for every integer $m\in \ZZ$, there
 exist integers $a<0$ and $b>0$ such that $ae-be'=m$. Then
 $f^ag^b\in A_m$ is a non-zero element, as desired.  

 The last assertion is straightforward.
\end{proof}

With the notation above, it follows from Lemma
\ref{lem:Real-structures} (iii) that
$X/\!/\mathbb{S}^{1}=(V/\!/\GM,\overline{\sigma})$ is a real affine
algebraic variety and that the morphism
$\pi:X=(V,\sigma) \rightarrow
X/\!/\mathbb{S}^{1}=(V/\!/\GM,\overline{\sigma})$ induced by the
inclusion $A_{0}\hookrightarrow A$ is an $\mathbb{S}^{1}$-invariant
morphism of real algebraic varieties, which is a categorical quotient
in the category of real affine algebraic varieties.

\section{Altmann-Hausen presentation of a circle action}
\label{section2}

The aim of this section is to establish a counter-part for real affine
varieties with circle actions of the geometrico-combinatorial
presentation of affine varieties with split tori actions developped
by Altmann and Hausen in \cite{AH}.

We first need to introduce a special kind of rational quotient for a
real affine algebraic variety endowed with an effective
$\mathbb{S}^{1}$ which is the real counterpart of the quotient
constructed in \cite{AH} for normal affine varieties with split tori
actions. Keeping the same notation as in $\S$ \ref{affine-case} above, we let $d>0$ be minimal
such that $\bigoplus_{m\in \mathbb{Z}}A_{dm}$ is generated by
$A_{\pm d}$ as a graded $A_{0}$-algebra.  By virtue of Lemma
\ref{lem:Charac-S1-actions}~$(ii)$ and $(iii)$, the closed sub-scheme $W$ of
$Y_0(V)=V/\!/\GM=\mathrm{Spec}(A_{0})$ with defining ideal
$I=\langle A_{d}\cdot A_{-d}\rangle\subset A_{0}$ is
$\overline{\sigma}$-invariant.

\begin{defn} \label{def:AH-quotient} The \emph{real AH-quotient} of
  $X=(V,\sigma,\mu)$ by the $\mathbb{S}^{1}$-action $\mu$ is the real
  quasi-projective variety formed by total space of the blow-up
  $\pi:Y(V)\rightarrow Y_0(V)$ of $Y_0(V)$ with center at $W$, endowed
  with the lift $\tau$ of the real structure $\overline{\sigma}$.
\end{defn}
 
The complex variety $Y(V)$ is semi-projective, and by virtue of
\cite[Theorem 1.9]{Tha96} (see also \cite{Pe15}), it is isomorphic to
the irreducible component of the fiber product
\[\mathrm{Proj}(\bigoplus_{m\leq
    0}A_{m})\times_{Y_0(V)}\mathrm{Proj}(\bigoplus_{m\geq 0}A_{m})\]
which dominates $Y_0(V)$. So $Y(V)$ coincides with the AH-quotient of
$V$ for the $\GM$-action $\mu$ in the sense of \cite{AH}.  By
construction, $\pi: (Y(V),\tau)\rightarrow (Y_0(V),\overline{\sigma})$
is a morphism of real algebraic varieties.

\subsection{Proper hyperbolic segmental divisors}
\label{Phs-div}
Recall that a Weil $\QQ$-divisor $D$ on a normal algebraic variety $Y$
is called $\QQ$-Cartier if $nD$ is Cartier for some $n\geq 1$. Furthermore, $D$ is 
semiample if there exits $n\geq 1$ such that the linear system $|nD|$ is base point free, 
equivalently such that the sheaf $\OO_Y(nD)$ is invertible and globally generated. The divisor $D$ is called big if there exists $E\in |nD|$ with affine complement for some $n\geq 1$. In the sequel, all our divisors will be
$\QQ$-Cartier Weil $\QQ$-divisors. We will refer to such divisors simply as $\QQ$-Cartier divisors.

\begin{notation}
  Given a $\mathbb{Q}$-Cartier divisor $D$ on a normal variety $Y$, we
  denote the round-down of $D$ by $\lfloor D\rfloor$. We identify the
  $\Gamma(Y,\mathcal{O}_{Y})$-module   $\Gamma(Y,\mathcal{O}_{Y}(\lfloor D\rfloor))$ with the
  sub-$\Gamma(Y,\mathcal{O}_{Y})$-module of the field of rational functions $\mathrm{Frac}(Y)$ of $Y$ 
  generated by rational functions $g\in\mathrm{Frac}(Y)$ such that
  $\mathrm{div}(g)+\lfloor D\rfloor\geq 0$. Under this identification, a rational function 
  $g\in\mathrm{Frac}(Y)$ satisfies $\mathrm{div}(g)+\lfloor D\rfloor\geq 0$ if and only if
  $\mathrm{div}(g)+ D\geq 0$. So we can set without ambiguity
  $\Gamma(Y,\mathcal{O}_{Y}(\lfloor D\rfloor)):=\Gamma(Y,\mathcal{O}_{Y}(D))$ and remove the round-down
  brackets from the notation.
\end{notation}

We now review a simple special case adapted to our context of the
general notion of polyhedral divisor defined in \cite{AH}.  Let
$N\simeq \mathbb{Z}$ be a rank one lattice and let $M$ be its dual.
Let $\mathcal{J}$ be the set of all closed intervals $[a,b]$ of
$N\otimes_{\mathbb{Z}}\mathbb{R}\simeq \mathbb{R}$ with rational
bounds, where we admit singleton intervals with $a=b$ as
$[a,a]=\{a\}$. The set $\mathcal{J}$ has the structure of an abelian
semi-group for the Minkowski sum $[a,b]+[a',b']=[a+a',b+b']$, with
identity $[0,0]=\{0\}$. Every element $m\in M$ determines a semi-group
homomorphism
\[\mathrm{ev}_{m}:\mathcal{J}\rightarrow\mathbb{Q},\quad
[a,b]\mapsto\min(ma,mb)=
\begin{cases}
  ma & \mbox{if } m\geq 0 \\
  mb & \mbox{if } m<0
\end{cases}
\,.\]

\begin{defn} 
  A \emph{segmental divisor} \cite{Pe15} on a normal algebraic variety
  $Y$ is an element
  \[\mathcal{D}=\sum[a_{i},b_{i}]\otimes D_{i} \in
    \mathcal{J}\otimes_{\mathbb{Z}}\mathrm{WDiv}(Y)\] of the
  semi-group of formal finite sums with coefficients in $\mathcal{J}$
  of prime Weil divisors on $Y$.  A segmental divisor
  $\mathcal{D}=\sum[a_{i},b_{i}]\otimes D_{i}$ is called \emph{proper}
  if for every $m\in\mathbb{Z}$, the Weil $\mathbb{Q}$-divisor
\[\mathcal{D}(m):=(\mathrm{ev}_{m}\otimes\mathrm{id})(\mathcal{D})=\sum\min(ma_i,mb_i)D_{i}\]
is a big, semiample $\mathbb{Q}$-Cartier divisor on $Y$.   
\end{defn}

Every Weil divisor $D$ on $Y$ determines a segmental divisor
$\{1\}\otimes\ D$, in particular every non-zero rational function $f$
on $Y$ determines a \emph{principal segmental divisor}
$\{1\}\otimes\mathrm{div}(f)$.  Note in addition that the definition
of the evaluation homomorphisms $\mathrm{ev}_{m}$ guarantees that for
every segmental divisor $\mathcal{D}$ and every pair of integers
$m,n\in \mathbb{Z}$ the Weil $\QQ$-divisors
$\mathcal{D}(m)+\mathcal{D}(n)-\mathcal{D}(m+n)$ are all
anti-effective. In particular,
$\mathcal{D}(m)+ \mathcal{D}(-m)\leq \mathcal{D}(0)=0$ for all
$m\in \mathbb{Z}$.
\begin{defn} %
  Given a dominant rational map $\psi:Y'\dashrightarrow Y$ between
  normal algebraic varieties and a segmental divisor
  $\mathcal{D}=\sum[a_{i},b_{i}]\otimes D_{i}$ on $Y$, the
  \emph{pull-back of $\mathcal{D}$ by $\psi$} is the segmental divisor
  $\psi^*\mathcal{D}:=\sum[a_{i},b_{i}]\otimes \psi^*D_{i}$ on $Y'$
  where for every $i$, $\psi^*D_{i}$ is the usual pull-back on $Y'$ of
  the Weil divisor $D_i$ on $Y$ by $\psi$.
\end{defn} 

Recall that the real structure $\rho$ on $\GM$ is given as the
composition of the automorphism induced by the involution
$-\mathrm{id}_M$ of its character lattice $M\simeq \ZZ$ with the
complex conjugation. The dual involution
$-\mathrm{id}_N=(-\mathrm{id}_M)^*$ of the lattice $N\simeq\ZZ$ of
$1$-parameter subgroups of $\GM$ induces an involution of
$\mathcal{J}$. Moreover, when $Z$ is a normal real algebraic variety
represented by a complex variety $Y$ with real structure $\tau$, the
pull-back of Weil $\QQ$-divisors on $Y$ by the real structure $\tau$
induces an involution $\tau^*$ on $\QQ$-divisors on $Y$. Putting these
two involutions together, we obtain an involution
\[(-\mathrm{id}_M)^*\otimes \tau^*:
  \mathcal{J}\otimes_{\mathbb{Z}}\mathrm{WDiv}(Y)
  \rightarrow\mathcal{J}\otimes_{\mathbb{Z}}\mathrm{WDiv}(Y), \quad
  \mathcal{D}=\sum[a_{i},b_{i}]\otimes D_{i}\mapsto
  \sum[-b_i,-a_{i}]\otimes\tau^{*}D_{i}.\]
 
\begin{defn} 
  A \emph{proper hyperbolic segmental pair (phs-pair)} on
  normal real algebraic variety $Z=(Y,\tau)$ is a pair
  $(\mathcal{D},h)$ consisting of a proper segmental divisor
  $\mathcal{D}=\sum[a_{i},b_{i}]\otimes D_{i}$ and a $\tau$-invariant rational function $h$ on $Y$
  such that
  \begin{align}
    \label{eq:flip}
  ((-\mathrm{id}_M)^*\otimes\tau^*)(\DD)=\DD+\{1\}\otimes\mathrm{div}(h).    
  \end{align}
 
  It will be convenient in practice to separate the homomorphism
  coming from the real structure on $Y$ from that coming from the real
  structure on $\GM$. So, up to changing $h$ for $h^{-1}$, we can
  rewrite \eqref{eq:flip} equivalently as
  \[\tau^*(\DD)=\flip{\DD}+\{1\}\otimes\mathrm{div}(h),\] 
  where
  $\flip{\DD}=((-\mathrm{id}_M)^*\otimes\mathrm{id})(\DD)=\sum[-b_{i},-a_{i}]\otimes
  D_{i}$.
\end{defn}

Remark that our definition of phs-pair agrees with
\cite[Definition~5.8]{La15} in the particular case where $Z$ is a
normal real curve.

\begin{lem}\label{lem:phs-basics} Let $(\mathcal{D},h)$ be a phs-pair
  on a normal semi-projective real algebraic variety
  $Z=(Y,\tau)$. Then the following hold:
  \begin{enumerate}[(i)]
\item The sheaf $\mathcal{O}_Y(\mathcal{D}(m))$ has a nonzero global section for all  $m\in \mathbb{Z}$.
  \item The real structure $\tau$ induces an isomorphism of
    $\Gamma(Y,\mathcal{O}_Y)$-modules
\begin{equation}
  \tau_{m}^{*}:\Gamma(Y,\mathcal{O}_{Y}(\mathcal{D}(m)))\stackrel{\simeq}{\longrightarrow}\Gamma(Y,\mathcal{O}_{Y}(\mathcal{D}(-m))),\quad
  g\mapsto h^{m}\cdot \tau^{*}g, \quad \mbox{for all } m \in \mathbb{Z}.\label{eq:graded-real-involution}
\end{equation}
Furthermore, $\tau_{0}^*=\Gamma(\tau)^{*}$ and $\tau_{-m}^*\circ\tau_{m}^*=\mathrm{id}$.

  \end{enumerate}
\end{lem}

\begin{proof} 
  The first assertion, in analogy with
  Lemma~\ref{lem:Charac-S1-actions}~(i), is again a standard fact
  for proper segmental divisors. Since $\mathcal{D}(\pm 1)$ is a big,
  there exist relatively prime positive integers $e$ and $e'$ and
  non-zero global sections $f$ and $g$ of
  $\mathcal{O}_Y(\mathcal{D}(e))$ and
  $\mathcal{O}_Y(\mathcal{D}(-e'))$, respectively. Now, for every
  integer $m\in \ZZ$, there exists integers $a<0$ and $b>0$ such that
  $ae-be'=m$. Then $f^ag^b$ is non-zero global section of
  $\mathcal{D}(ae-be')=\mathcal{D}(m)$ as desired.
 
  For the second assertion, the fact that $\tau_{0}^*=\Gamma(\tau)^{*}$ follows from the definition
  of the real structure $\Gamma(\tau)$ on   $\mathrm{Spec}(\Gamma(Y,\mathcal{O}_{Y})$ (see $\S$ \ref{section2}). Given
  $m\in \mathbb{Z}\setminus\{0\}$, it follows from the definition of a
  phs-pair that   $\tau^*(\mathcal{D})(m)=\mathcal{D}(-m)+m\divi(h)$. This implies
  that the homomorphism $\tau_{m}^{*}$ in   \eqref{eq:graded-real-involution} is well-defined. Since $\tau^{*}$
  is an involution, $\tau_{-m}^{*}\circ \tau_{m}^{*}$ is the identity
  of $\Gamma(Y,\mathcal{O}_{Y}(\mathcal{D}(m)))$, and so   $\tau_{m}^{*}$ is an isomorphism with inverse $\tau_{-m}^{*}$. 
\end{proof}

\subsection{Real Altmann-Hausen presentations}

In this section, we state and prove our main theorem giving a
geometrico-combinatorial presentation of normal real affine varieties endowed
with a circle action.

Let $Y$ be a normal semi-projective complex variety and let
$\mathcal{D}$ be a proper segmental divisor on $Y$. Then it follows
from \cite[Theorem 3.1]{AH} that the $\mathbb{C}$-scheme
\[V=V(Y,\mathcal{D}):=\mathrm{Spec}\left(\bigoplus_{m\in\mathbb{Z}}
    \Gamma\left(Y,\mathcal{O}_{Y}(\mathcal{D}(m))\right)\right)
  \]
  is a normal complex affine variety of dimension $\dim
  Y+1$. Furthermore, the $\mathbb{Z}$-grading of its coordinate ring
  uniquely determines an effective $\GM$-action
  $\mu:\GM\times V\rightarrow V$ with algebraic quotient isomorphic to
  $\mathrm{Spec}(\Gamma(Y,\mathcal{O}_{Y})$ and whose AH-quotient is
  birationally dominated by $Y$.

  \begin{thm} \label{thm:real-AH-Main} %
    Let $Z=(Y,\tau)$ be a normal semi-projective real algebraic
    variety, represented by a complex variety $Y$ with real structure
    $\tau$, and let $(\mathcal{D},h)$ be a phs-pair on $Z$. Then the
    following hold:
    \begin{enumerate}[(1)]
    \item The normal complex affine variety $V=V(Y,\mathcal{D})$
      carries a real structure $\sigma$ such that
      $\sigma\circ\mu=\mu\circ(\rho\times\sigma)$.
  
    \item The triple $(V,\sigma, \mu)$ is a normal real affine
      algebraic variety $X(Z,(\mathcal{D},h))$ of dimension $\dim Z+1$
      endowed with an effective $\mathbb{S}^{1}$-action with algebraic
      quotient  $(\mathrm{Spec}(\Gamma(Y,\mathcal{O}_{Y}),\Gamma(\tau))$ and
      real AH-quotient birationally dominated to $Z$.

 \item Conversely, every normal real affine variety $X=(V,\sigma,\mu)$
   endowed with a effective $\mathbb{S}^{1}$-action is equivariantly
   isomorphic to $X(Z,(\mathcal{D},h))$ for a suitable phs-pair
   $(\mathcal{D},h)$ on its real AH-quotient $Z=(Y(V),\tau)$.

 \end{enumerate}
\end{thm}

\begin{proof}  
  Since $(\mathcal{D},h)$ is a phs-pair, it follows from Lemma
  \ref{lem:phs-basics} (ii) that there exist isomorphisms
  $\tau_m^{*}:\Gamma(Y,\mathcal{O}_{Y}(\mathcal{D}(m)))
  \stackrel{\simeq}{\longrightarrow}
  \Gamma(Y,\mathcal{O}_{Y}(\mathcal{D}(-m)))$ for every
  $m\in \mathbb{Z}$. These collect into an involution
  $\sigma^{*}=\bigoplus_{m\in\mathbb{Z}}\tau_{m}^{*}$ on the direct
  sum
  $A=\bigoplus_{m\in\mathbb{Z}}\Gamma(Y,\mathcal{O}_{Y}(\mathcal{D}(m)))$. The
  latter corresponds to a real structure $\sigma$ on $V$ such that by
  construction $\sigma\circ\mu=\mu\circ(\rho\times\sigma)$. It
  then follows from Lemma \ref{lem:Charac-S1-actions} that
  $(V,\sigma,\mu)$ represents a normal real affine variety $X$ endowed
  with an effective $\mathbb{S}^{1}$-action. The facts that
  $X/\!/\mathbb{S}^{1}\simeq
  \mathrm{Spec}(\Gamma(Y,\mathcal{O}_{Y}),\Gamma(\tau))$ and that the
  real AH-quotient of $X$ is birationnally dominated by $Z$ follow
  from the corresponding assertions for $V$ and the construction of
  $\sigma$. This proves 1) and 2).

For the converse 3), let $A=\bigoplus_{m\in \mathbb{Z}} A_m$ be the decomposition of the coordinate ring of $V$ into 
semi-invariants sub-spaces for the $\GM$-action $\mu$ and let $(Y_0=\mathrm{Spec}(A_0),\overline{\sigma})$ and $(Y,\tau)$ be the algebraic quotient and the real AH-quotient of $X$ respectively. By construction, $Y$ is birational to $Y_0$, and we can therefore identify its field of rational functions with the field of fractions $\mathrm{Frac}(A_0)$ of $A_0$. By Lemma \ref{lem:Charac-S1-actions} (i), $\mu$ admits a semi-invariant regular function $s$ of weight $1$. Then \cite[Theorem 3.4]{AH} guarantees the existence of a proper segmental divisor $\mathcal{D}$ on $Y$ such that for every $m\in \mathbb{Z}$ the sub-$A_0$-module $s^{-m}A_m$ of $\mathrm{Frac}(A_0)$ is equal to $\Gamma(Y,\mathcal{O}_{Y}(\mathcal{D}(m)))$ and such that $A$ is equal to 
$\bigoplus_{m\in\mathbb{Z}}\Gamma(Y,\mathcal{O}_{Y}(\mathcal{D}(m))\cdot s^m$ as a graded sub-${A}_0$-algebra of $\mathrm{Frac}(A_0)(s)$. 

We will now check that $h=s\sigma^*(s)$ is a $\tau$-invariant rational
function on $Y$ making $(\mathcal{D},h)$ a phs-pair. Note that by
construction, $h$ is a $\sigma$-invariant rational function on
$V$. Since by Lemma \ref{lem:Charac-S1-actions} (ii), $\sigma^*(s)$ is
a semi-invariant regular function of weight $-1$, it follows that $h$
is a also $\GM$-invariant, hence an element of
$\mathrm{Frac}(A_0)$. Since $\sigma^*$ coincides by definition with
$\overline{\sigma}$ on $A_0$ and $\tau$ is lifted from
$\overline{\sigma}$, we can thus view $h$ as a $\tau$-invariant
rational function on $Y$. Let $m\geq 1$ be such that $\DD(m)$ and
$\DD(-m)$ are Cartier and globally generated, and let
$\{U_i\}_{i\in I}$ be an open cover of $Y$ such that $\DD(m)$ and
$\DD(-m)$ are principal in every $U_i$.  Then, by the construction of
$\mathcal{D}$ in \cite[Theorem 3.4]{AH}, for every $i\in I$ there
exists a rational function $g \in \mathrm{Frac}(A_0)$ such that
$gs^m\in A_m$ and $\divi(g)|_{U_i}+\DD(m)|_{U_i}=0$. In particular, we
have
$\divi(\tau^*g)|_{\tau^{-1}(U_i)}=-\tau^*\DD(m)|_{\tau^{-1}(U_i)}$. Since
$\sigma^*(gs^m)=(\tau^*g)h^ms^{-m}\in A_{-m}$, it follows from the construction of $\DD$ that 
$(\tau^*g)h^m\in \Gamma(Y,\mathcal{O}_{Y}(\mathcal{D}(-m)))$ and hence, we have
$$\divi(\tau^*g)+m\divi(h)+\DD(-m)\geq 0.$$ 
Substituting $\divi(\tau^*g)|_{\tau^{-1}(U_i)}=-\tau^*\DD(m)|_{\tau^{-1}(U_i)}$, we
conclude that
$$(-\tau^*\DD(m)+m\divi(h)+\DD(-m))|_{\tau^{-1}(U_i)}\geq 0.$$ Since
$\{\tau^{-1}(U_i)\}_{i\in I}$ is also an open cover of $Y$, this
inequality is independant on the open subset $U_i$, and we obtain
\begin{align}
  \label{eq:main-th-1}
  -\tau^*\DD(m)+ m \divi(h)+\DD(-m)\geq 0
\end{align}
for every $m\in \mathbb{Z}$ since all the terms in \eqref{eq:main-th-1} are
linear on $m$. 
Taking now a rational function $g\in \mathrm{Frac}(A_0)$ such that $gs^{-m}\in A_{-m}$ and
$\DD(-m)|_{U_i}=-\divi(g)|_{U_i}$, we find by the same argument that $-\tau^*\DD(-m)-m\divi(h)+\DD(m)\geq 0$, hence applying the involution $\tau^*$ to this inequality that 
\begin{align}
  \label{eq:main-th-2}
  -\tau^*\DD(m)+m\divi(h)+\DD(-m)\leq 0
\end{align}
for every $m\in \mathbb{Z}$. The inequalities \eqref{eq:main-th-1} and \eqref{eq:main-th-2} together yield that
$\tau^*(\DD)(m)=\DD(-m)+m\divi(h)$ for every $m\in \mathbb{Z}$, hence that 
$\tau^{*}\mathcal{D}=\flip{\mathcal{D}}+\{1\}\otimes\divi(h)$.
\end{proof}

\begin{rem} \label{rem:semi-inv-w1} Given a normal real affine variety
  $X=(V,\sigma,\mu)$ endowed with a effective $\mathbb{S}^{1}$-action,
  the construction of a proper segmental divisor $\mathcal{D}$ on the
  AH-quotient $Y(V)$ such that $V$ is $\GM$-equivariantly isomorphic
  to $V=V(Y,\mathcal{D})$ depends on the non-canonical choice of a
  semi-invariant rational function $s$ of weight $1$ on $V$. Different
  choices for $s$ lead of course to different segmental divisors
  $\mathcal{D}_s$ on $Y(V)$. In the proof of Theorem
  \ref{thm:real-AH-Main}, we made a particular choice for $s$, but the
  proof actually shows that for every other choice of $s$, the
  rational function $h_s=s\sigma^*(s)$ on $Z=(Y(V),\tau))$ is
  $\tau$-invariant and $(\mathcal{D}_s,h)$ is a phs-pair on $Z$.
\end{rem}

\begin{defn} A couple consisting of a real normal semiprojective
  variety $Z=(Y,\tau)$ and a phs-pair $(\DD,h)$ on it is called
  \emph{minimal} if $Z$ is the real AH-quotient of $X(Z,(\DD,h))$.
\end{defn}

It follows from the definition of the real AH-quotient of
$X(Z,(\DD,h))$ that a couple $(Z,(\DD,h))$ is minimal if and only if
$Y$ is the AH-quotient of $V(Y,\DD)$. This definition is thus
equivalent to requiring that the couple $(Y,\DD)$ is minimal in the
sense of \cite[Definition~8.7]{AH}.

\begin{cor}\label{cor:existence-h}
  Let $X=(V,\sigma,\mu)$ be a normal real affine variety endowed with
  an effective $\mathbb{S}^{1}$-action and let $\DD$ be any proper
  segmental divisor on a complex semiprojective variety $Y$ such that
  there exists a $\GM$-equivariant isomorphism
  $\phi:V\stackrel{\simeq}{\rightarrow} V(Y,\DD)$.

  If the real structure $\overline{\sigma}$ on
  $V/\!/\GM \simeq V(Y,\DD)/\!/\GM$ induced by $\sigma$ lifts to a
  real structure $\tau$ on $Y$, which holds for instance if $(Y,\DD)$
  is minimal, then there exists a $\tau$-invariant rational function
  $h$ on $Z=(Y,\tau)$ making $(Z,(\DD,h))$ a minimal phs-pair such
  that $X(Z,(\DD,h))$ is $\mathbb{S}^{1}$-equivariantly isomorphic to
  $X$.
\end{cor}

\begin{proof}
  This follows from the same argument as in the proof of
  Theorem~\ref{thm:real-AH-Main} (3), taking $Y$ as in the corollary
  instead of the AH-quotient $Y(V)$ of $V$.
\end{proof}

\begin{thm} \label{prop:Equiv-Iso} %
  For $i=1,2$, let $Z_i=(Y_i,\tau_i)$ be normal real semiprojective
  varieties and let $X_i=X(Z_i,(\mathcal{D}_i,h_i))$ be normal real
  affine varieties with effective $\mathbb{S}^1$-actions determined by
  respective phs-pairs $(\mathcal{D}_i,h_i)$ on $Z_i$.

  Then $X_1$ and $X_2$ are $\mathbb{S}^1$-equivariantly isomorphic and
  only if there exits a third normal real affine variety $X=X(Z,(\DD,h))$
  with effective $\mathbb{S}^1$-action for a certain phs-pair
  $(\DD,h)$ on a normal real semiprojective variety $Z=(Y,\tau)$, real birational morphisms
  $\psi_i:Z_i\rightarrow Z$ and rational functions $f_i$ on $Y_i$,
  $i=1,2$, such that
   \[ \psi_i^{*}(\mathcal{D})=\mathcal{D}_i+\{1\}\otimes\mathrm{div}(f_i) \quad \textrm{and} \quad \psi_i^*(h)=(f_i\cdot\tau_i^*f_i)\cdot h_i.\]
\end{thm}

\begin{proof}
  Recall that by definition, $X_i=(V_i,\sigma_i,\mu_i)$ where
  $V_i=V(Y_i,\DD_i)$ and $\sigma_i$ is the real structure on $V_i$
  constructed in Theorem~\ref{thm:real-AH-Main} (1). Assume first that
  $X_1$ and $X_2$ are $\mathbb{S}^1$-equivariantly isomorphic, let
  $Z=(Y,\tau)$ be the AH-quotient of $V_1$ endowed with the real
  structure induced by $\sigma_1$. By \cite[Definition~8.7]{AH}, there
  exists a proper segmental divisor $\DD$ on $Y$ such that
  $V=V(Y,\DD)$ is $\GM$-equivariantly isomorphic to $V_1$. It then
  follows from Corollary~\ref{cor:existence-h} that there exists a
  $\tau$-invariant function $h$ on $Y$ such that $(\DD,h)$ is phs-pair
  and $X(Z,(\DD,h))=(V,\sigma,\mu)$ is $\mathbb{S}^1$-equivariantly
  isomorphic to $X_1$. By \cite[Theorem~8.8]{AH}, there exist
  birational morphisms $\psi_i:Y_i\rightarrow Y$ such that
  $\psi_i^{*}(\mathcal{D})=\mathcal{D}_i+\{1\}\otimes\mathrm{div}(f_i)$
  for some rational functions $f_i$ on $Y_i$. Furthermore, the real
  structures $\tau_i$ on $Y_i$ and $\tau$ on $Y$ are all lifts of the
  real structure $\overline{\sigma}_1$ on the algebraic quotient $V_1/\!/\GM$, which is
  birational to $Y_i$ and $Y$ since the $\GM$-actions considered are
  all hyperbolic by Lemma \ref{lem:Charac-S1-actions} (i). Hence, we
  have $\tau\circ\psi_i=\psi_i\circ\tau_i$ which shows that $\psi_i$
  is a real birational morphism. By \cite[Proposition~8.6]{AH}, the
  isomorphism between $V$ and $V_i$ is given by the collection of
  isomorphisms
  \begin{align}
    \label{eq:iso}
    \Psi_i^*:\Gamma(Y,\OO_Y(\DD(m)))\stackrel{\simeq}{\longrightarrow}
    \Gamma(Y_i,\OO_{Y_i}(\DD_i(m))),\quad g\mapsto f_i^{-m}\cdot \psi_i^*(g).
  \end{align}
  Let $s$ and $s_i$ be the regular functions on $V$ and $V_i$
  corresponding respectively to $1$ in degree $1$ in the grading of
  their coordinate rings by the subspaces $\Gamma(Y,\OO_Y(\DD(m)))$
  and $\Gamma(Y,\OO_Y(\DD_i(m)))$. By construction, we have
  $h=s\sigma^*(s)$ and $h_i=s_i\sigma_i^*(s_i)$. Since on the other
  hand $\Psi_i^*(s)=f_is_i$, we have
  \begin{align*}
    \psi_i^*(h)&=\psi_i^*(s\sigma^*(s))=\Psi_i^*(s)\cdot
                 \Psi_i^*\circ\sigma^*(s)=\Psi_i^*(s)\cdot
                 \sigma_i^*\circ\Psi_i^*(s)\\
               &=f_is_i\cdot
                 \sigma_i^*(f_i)\sigma_i^*(s_i) =f_i\cdot\sigma^*_i(f_i)\cdot h_i=(f_i\cdot\tau_i^*f_i)\cdot h_i.
  \end{align*}

  We now prove the converse statement. Let $i=1$ or $i=2$. By
  \cite[Theorem~8.8]{AH}, $V$ is isomorphic to $V_i$ and the
  isomorphism is given by \eqref{eq:iso}. The real structures $\sigma$
  on $V$ and $\sigma_i$ on $V_i$ are given as in the proof of
  Theorem~\ref{thm:real-AH-Main} (1) via the collection of
  isomorphisms \eqref{eq:graded-real-involution}. To conclude that $X$
  is $\mathbb{S}^1$-equivariantly isomorphic to $X_i$ we only need to
  check that $\sigma^*_i\circ\Psi^*_i=\Psi^*_i\circ\sigma^*$. But for
  every $g\in \Gamma(Y,\OO_Y(\DD(m)))$, we have
  \begin{align*}
    \Psi^*_i\circ\sigma^*(g)&=\Psi^*_i(h^m\tau^*(g))
                              =f_i^{-m}\cdot\psi^*_i(h^m)\cdot\psi^*_i\tau^*(g)
                              =f_i^{-m}\cdot f_i^m
                              \cdot\tau_i^*(f_i^m)\cdot
                              h_i^m\cdot\psi^*_i\tau^*(g) \\
                            &=\tau_i^*(f_i^m)\cdot
                              h_i^m\cdot\tau_i^*\psi^*_i(g)=\sigma^*_i(f_i^m\psi^*_i(g))=\sigma^*_i\circ\Psi^*_i(g),
  \end{align*}
  which concludes the proof.
\end{proof}

\begin{cor} \label{cor:AH-iso-same} Let $(Z_i,(\DD_i,h_i))$ be minimal
  couples on real normal semiprojective varieties $Z_i=(Y_i,\tau_i)$,
  $i=1,2$, determining normal real affine varieties
  $X_i=X(Z_i,(\mathcal{D}_i,h_i))$ with effective
  $\mathbb{S}^{1}$-actions.  Then $X_1$ and $X_2$ are
  $\mathbb{S}^1$-equivariantly isomorphic if and only if there exists
  a real isomorphism $\psi:Z_1\rightarrow Z_2$ and a rational function
  $f_1$ on $Y_1$ such that
  \[ \psi^{*}(\DD_2)=\mathcal{D}_1+\{1\}\otimes\mathrm{div}(f_1) \quad
    \textrm{and} \quad \psi^*(h_2)=(f_1\cdot\tau_1^*f_1)\cdot h_1.\]
\end{cor}

\begin{proof}
  This follows directly from \cite[Theorem~8.8]{AH} which asserts that
  $\psi_1$ and $\psi_2$ in Theorem~\ref{prop:Equiv-Iso} are both
  isomorphisms.
\end{proof}

\begin{cor} \label{cor:real forms}
  Let $(Z,(\DD,h))$ be a minimal couple on a real normal semiprojective variety $Z=(Y,\tau)$ determining a normal real affine variety $X=X(Z,(\DD,h))$ with an effective $\mathbb{S}^{1}$-action. Then every real form of the $\mathbb{S}^1$ variety $X$ is $\mathbb{S}^1$-equivariantly isomorphic to $X(Z',(\DD,h'))$ form some real form $Z'=(Y,\tau')$ of $Z$ and a $\tau'$-invariant rational function $h'$ on $Y$ making $(\DD,h')$ a phs-pair on $Z'$.
\end{cor}
\begin{proof}
  Recall that by definition $X=(V,\sigma,\mu)$ where
  $V=V(\mathcal{D},h')$ is endowed with the $\GM$-action $\mu$ given
  by the $\mathbb{Z}$-grading of its coordinate ring, and $\sigma$ is
  the real structure on $V$ constructed in the proof of Theorem
  \ref{thm:real-AH-Main} (1). On the other hand, by Theorem
  \ref{thm:real-AH-Main} (3), every real form
  $X_1=(V_1,\mu_1,\sigma_1)$ of $X$ is $\mathbb{S}^1$ equivariantly
  isomorphic to $X(Z_1,(\mathcal{D}_1,h_1))$ for a suitable phs-pair
  $(\DD_1,h_1)$ on its real AH-quotient $Z_1=(Y_1,\tau_1)$. Since by
  definition $V_1$ is $\GM$-equviariantly isomorphic to $V$ and the
  couple $(Y,\DD)$ is minimal, it follows \cite[Theorem 8.8]{AH} that
  there exist an isomorphism of complex varieties
  $\psi:Y\rightarrow Y_1$ such that $\DD=\psi^*\DD_1+\divi(f)$ for
  some rational function $f$ on $Y$. Letting $\tau'=\psi^*(\tau_1)$,
  the rational function
  $h'=f(\psi^{-1}\tau_{1}\psi(f))^{-1}\psi^{*}h_{1}$ on $Y$ is
  $\tau'$-invariant, and $(\DD,h')$ is phs pair on $Z'=(Y,\tau')$ such
  that $X_1$ is $\mathbb{S}^1$-equivariantly isomorphic to
  $X(Z',(\DD,h'))$.
\end{proof}

\begin{rem} A direct adaptation of \cite[Section~8]{AH} in our context
  provides a notion of \emph{morphism of phs-pairs} for which the
  appropriate extension of Theorem~\ref{prop:Equiv-Iso} yields that
  the assignment $(Z,(\DD,h))\mapsto X(Z,(\DD,h))$ is a faithful
  covariant functor from the category of phs-pairs on normal real 
  semiprojective varieties to the category of normal real affine
  varieties with effective $\mathbb{S}^1$-actions.
\end{rem}

\subsection{Real DPD presentations}\label{rem:seg-div-couple}

It is well-known that proper segmental divisors can be described by a simpler datum consisting
of a suitable pair of $\QQ$-divisors. Indeed, given a proper segmental divisor $\mathcal{D}$
on a normal complex algebraic variety $Y$, we let
$D_{+}=\mathcal{D}(1)$ and $D_{-}=\mathcal{D}(-1)$. The identity
    \begin{align}
    \label{eq:DPD}
    \mathcal{D}=\{1\}\otimes D_+ + [0,1]\otimes(-D_+ - D_-)  
  \end{align}
  implies that $\mathcal{D}$ is equivalently fully determined by a
  couple of big and semiample $\mathbb{Q}$-Cartier divisors $D_{+}$
  and $D_{-}$ on $Y$ satisfying $D_++D_-\leq 0$ (see $\S$
  \ref{Phs-div}).  Furthermore, if $(\mathcal{D},h)$ is a phs-pair for
  a given additional real structure $\tau$ on $Y$, then the identity
  $\tau^{*}\mathcal{D}=\flip{\mathcal{D}}+\{1\}\otimes\mathrm{div}(h)$
  is equivalent to $D_{-}=\tau^{*}D_{+}-\mathrm{div}(h)$.  Summing up,
  a phs-pair $(\mathcal{D},h)$ on a normal real algebraic variety
  $Z=(Y,\tau)$ is equivalently fully determined via \eqref{eq:DPD} by
  a pair $(D,h)$ satisfying $D+\tau^*D\leq \divi(h)$. The original
  data is recovered from \eqref{eq:DPD} by setting $D_+=D$ and
  $D_-=\tau^{*}D-\mathrm{div}(h)$.  By analogy with the terminology
  introduced by Flenner and Zaidenberg \cite{FZ03} for the description
  of normal complex affine surfaces with $\GM$-actions, we set the
  following definition:
\begin{defn}  
  A \emph{real DPD pair} on a normal real algebraic variety
  $Z=(Y,\tau)$ is a pair $(D,h)$ consisting of a big and semiample
  $\mathbb{Q}$-Cartier divisor $D$ and a $\tau$-invariant
  rational function $h$ on $Y$ satisfying $D+\tau^*D\leq \divi(h)$.
\end{defn}

Here, DPD stands for Dolgachev, Pinkham and Demazure, respectively who
where the first to describe split $\GM$-actions via 
$\QQ$-divisors on their quotients, see the references in \cite{FZ03}. The
following corollary is a straightforward reformulation of Theorem
\ref{thm:real-AH-Main} and Corollary~\ref{cor:AH-iso-same} in terms of
DPD-pairs:
\begin{cor}
\label{cor:DPD-General} A normal real affine variety $X$ with an effective $\mathbb{S}^{1}$-action is determined by the following
data:
\begin{enumerate}[(1)]
\item A real normal semiprojective variety $Z=(Y,\tau)$
  representing the real AH-quotient of $X$,

\item A pair $(D,h)$ consisting of a big and semiample
  $\mathbb{Q}$-Cartier divisor $D$ and a $\tau$-invariant rational
  function $h$ on $Y$ satisfying $D+\tau^*D\leq \divi(h)$.
\end{enumerate}

Furthermore, for a fixed $Z=(Y,\tau)$, two pairs $(D_1,h_1)$ and $(D_2,h_2)$ determine $\mathbb{S}^1$-equivariantly isomorphic affine varieties if and
only if there exists a real automorphism $\psi$ of $Z$ and a rational
function $f$ on $Y$ such that \[ \psi^*D_2=D_1+\mathrm{div}(f) \quad \textrm{and} \quad
\psi^*(h_2)=(f\cdot\tau^*f)\cdot h_1.\]
\end{cor}

\begin{rem} \label{rem:Galois-Module} %
  Given a real algebraic variety $Z=(Y,\tau)$,  the group of $K(Y)^{*}$ of invertible rational functions on $Y$ has
  the structure of Galois module under the action of $\tau$. In terms of this structure, the condition
  $\psi^*(h_2)=(f\cdot\tau^*f)\cdot h_1$ in Corollary ~\ref{cor:AH-iso-same} and Corollary~\ref{cor:DPD-General} means that $h_1$ and $\psi^*(h_2)$ have the same class in the Galois cohomology group 
\[H^{2}(\mathrm{Gal}(\mathbb{C}/\mathbb{R}),K(Y)^{*})=(K(Y)^{*})^{\tau^*}/\mathrm{Im}(\mathrm{id}\times\tau^{*}).\]
\end{rem}

\section{Low dimensional examples}
\label{section3}

In this section we consider real affine curves and surfaces with $\mathbb{S}^{1}$-actions. 
In the surface case, rephrasing Corollary \ref{cor:DPD-General}, we obtain in particular a real counterpart of
Flenner-Zaidenberg DPD-presentation of normal complex affine surfaces with $\mathbb{G}_{m,\mathbb{C}}$-actions \cite{FZ03}. 
We illustrate the methods to explicitly find phs-pairs and DPD-pairs
corresponding to given $\mathbb{S}^{1}$-actions on various on examples.

\subsection{Real affine curves with $\mathbb{S}^{1}$-actions }

Let us first explain how to re-derive the following classical
characterization of real affine curves with a effective
$\mathbb{S}^{1}$-actions:

\begin{prop}
  Up to equivariant isomorphism there exists precisely two normal real
  affine curves with an effective $\mathbb{S}^{1}$-action:
  \begin{enumerate}[(a)]
  \item The circle
    $\mathbb{S}^{1}=\mathrm{Spec}(\mathbb{R}[x,y,]/(x^{2}+y^{2}-1)$
    acting on itself by translations

  \item The curve
    $C=\mathrm{Spec}(\mathbb{R}[u,v]/(u^{2}+v^{2}+1))
    \subset\mathbb{A}_{\mathbb{R}}^{2}$ on which $\mathbb{S}^{1}$ acts
    by restriction of the representation
    $\rho_{0}:\mathbb{S}^{1}\rightarrow\mathrm{SO}_{2,\mathbb{R}}$
    defined in $($\ref{eq:Representation}$)$.

\end{enumerate}
\end{prop}

\begin{proof}
Since the complex punctured affine line $\mathbb{A}_{*}^{1}=\mathrm{Spec}(\mathbb{C}[z^{\pm1}])$
is the only normal complex affine curve admitting an effective hyperbolic $\GM$-action, 
namely the one $\mu$ by translations $t\cdot z=tz$, a normal real affine curve $X$ endowed with an effective
$\mathbb{S}^{1}$-action is represented by a triple $(V=\mathbb{A}_{*}^{1},\sigma, \mu)$. Its real AH-quotient is thus isomorphic
to $\mathrm{Spec}(\mathbb{C})$, endowed with the complex conjugation, an a phs-pair $(\mathcal{D},h)$ on it consists of the trivial divisor and a non-zero real number $h\in\mathbb{R}^{*}$. By Theorem~\ref{prop:Equiv-Iso} and Remark \ref{rem:Galois-Module},
two real numbers $h$ and $h'$ determine $\mathbb{S}^{1}$-equivariantly isomorphic
curves if and only if they have the same class in $H^{2}(\mathbb{Z}_{2},\mathbb{C}^{*})\simeq\mathbb{R}/\mathbb{R}_{+}$,
that is, if and only if they have the same sign. We thus have two
cases:

\begin{enumerate}[(a)]
\item  $h=1$. The corresponding real structure $\sigma$ on $\mathbb{A}_{*}^{1}$ as constructed in Theorem \ref{thm:real-AH-Main} is given 
by the composition of the involution $z\mapsto z^{-1}$ with the complex conjugation. The invariants are then generated by $x=\frac{1}{2}(z+z^{-1})$ and $y=\frac{1}{2i}(z-z^{-1})$, and we conclude that $X=\mathrm{Spec}(\mathbb{C}[z^{\pm1}]^{\sigma^*})\simeq \mathbb{S}^{1}$
on which $\mathbb{S}^{1}$ acts by translations. 

\item $h=-1$. The corresponding real structure $\sigma$ is given by the composition of the involution $z\mapsto -z^{-1}$ with the complex conjugation
The invariants are generated by $u=\frac{1}{2}(z-z^{-1})$ and $v=\frac{1}{2i}(z+z^{-1})$, and the corresponding real affine curve $X=\mathrm{Spec}(\mathbb{C}[z^{\pm1}]^{\sigma^*})$ is isomorphic to $C$ with the announced $\mathbb{S}^{1}$-action.
\end{enumerate}
\end{proof}

\subsection{Real DPD-presentation of affine surfaces with $\mathbb{S}^{1}$-actions}

Given a normal real affine surface endowed with an effective
$\mathbb{S}^{1}$-action $X=(V,\sigma,\mu)$, the AH-quotient $Y(V)$ of
$V$ coincides with its algebraic quotient $Y_{0}(V)=V/\!/\GM$, which
is a normal, hence smooth complex affine curve.  The pair
$(Y_0(V),\overline{\sigma})$ is thus a smooth real affine
curve. Corollary \ref{cor:DPD-General} can be rephrased in this
case in the form of the following real counterpart of
Flenner-Zaidenberg DPD-presentation of normal complex affine surfaces
with $\mathbb{G}_{m,\mathbb{C}}$-actions \cite{FZ03}:

\begin{prop} \label{inter-kevin}
  A normal real affine surface $X$ with an effective
  $\mathbb{S}^{1}$-action is determined by a smooth real affine curve
  $C=(Y,\tau)$ and a pair $(D,h)$ consisting of a Weil
  $\mathbb{Q}$-divisor $D$ and a $\tau$-invariant rational function
  $h$ on $Y$ such that $D+\tau^*D\leq \divi(h)$.
\end{prop}

\begin{example} \label{ex:standard-conic-bundle} %

  Given a non-constant polynomial $P\in \mathbb{R}[w]$, we let $X(P)$
  be the normal real affine surface in  $\mathbb{A}^2_{\mathbb{R}}\times
  \mathbb{A}^1_{\mathbb{R}}=\mathrm{Spec}(\mathbb{R}[x,y][w])$ defined
  by the equation $x^2+y^2-P(w)=0$. The action of $\mathbb{S}^{1}$ on
  $\mathbb{A}^2_{\mathbb{R}}\times \mathbb{A}^1_{\mathbb{R}}$ defined
  by the direct sum of the representation
  $\rho_{0}:\mathbb{S}^{1}\rightarrow\mathrm{SO}_{2,\mathbb{R}}$ of
  \eqref{eq:Representation} on the first factor with the trivial
  representation on the second factor restricts to an effective
  $\mathbb{S}^{1}$-action on $X(P)$. We will show that a
  DPD-presentation for $X(P)$ is $(D,h)=(0,P(w))$ on the curve
  $C=\mathrm{Spec}(\mathbb{R}[w])=(\mathrm{Spec}(\mathbb{C}[w]),\tau)$,
  where $\tau$ is the complex conjugation.

  Indeed, by making the complex coordinate change $(u,v)=(x+iy,x-iy)$,
  we see that $X(P)$ endowed with its $\mathbb{S}^{1}$-action is
  represented by the triple $(V(P),\sigma,\mu)$ where $V(P)$ is the
  normal complex surface with equation $uv-P(w)=0$ in
  $\mathbb{A}^3_{\mathbb{C}}$, $\sigma$ is the real structure defined
  as the composition of the involution $(u,v)\mapsto (v,u)$ with the
  complex conjugation and $\mu$ is the effective $\GM$-action induced
  by the linear action $t\cdot(u,v,w)=(t^{-1}u,tv,w)$ on
  $\mathbb{A}^3_{\mathbb{C}}$. The quotient $V(P)/\!/\GM$ is
  isomorphic to $\mathrm{Spec}(\mathbb{C}[w])$ on which $\sigma$
  induces the complex conjugation $\tau$. Choosing $v$ as a
  semi-invariant function of weight $1$ on $V(P)$ as in the proof of Theorem
  \ref{thm:real-AH-Main}, we deduce from the identification:
  \[\Gamma(V(P),\mathcal{O}_{V(P)})\simeq \bigoplus_{n<0}\mathbb{C}[w]
    \cdot P(w)^{-n} v^n \oplus \mathbb{C}[w]\oplus \bigoplus_{n>0}
    \mathbb{C}[w] \cdot v^n \subset \mathbb{C}(w)(v) \] %
  that $V(P)$ is $\GM$-equivariantly isomorphic to
  $V(\mathrm{Spec}(\mathbb{C}[w]),\DD)$, where $\DD$ is determined by
  $\DD(n)=\operatorname{div}(1)=0$ and
  $\DD(-n)=\operatorname{div}(P(w)^{-n})$, for all $n>0$. We obtain from
  \eqref{eq:DPD} that
  \[\DD=\sum [0,p_i]\otimes \{a_i\} +\sum [0,q_i]\otimes
    (\{b_i\}+\{\overline{b}_i\})\,,\] %
  where $a_i$, $b_i$ and $\overline{b}_i$ are the real and complex
  roots of $P$ respectively, and $p_i,q_i$ are their respective multiplicities. Furthermore, $h=v\sigma^*v=vu=P(w)$,
  and so, the DPD-presentation of $X(P)$ is $(\DD(1),P(w))=(0,P(w))$ as claimed.
\end{example}

Note that in Example \ref{ex:standard-conic-bundle} above, the special
case where $P(w)=w$ corresponds to a surface $X(P)$ equivariantly
isomorphic to the affine plane $\mathbb{A}_{\mathbb{R}}^{2}$ endowed
with the effective $\mathbb{S}^{1}$-action defined by the
representation
$\rho_{0}:\mathbb{S}^{1}\rightarrow\mathrm{SO}_{2,\mathbb{R}}$. The
following is a counter-part for $\mathbb{S}^1$-actions of Gutwirth's
linearisation theorem \cite{Gu62} for $\mathbb{G}_m$-actions on the
plane:

\begin{prop} %
  Every effective $\mathbb{S}^{1}$-action on
  $\mathbb{A}_{\mathbb{R}}^{2}$ is conjugate by an automorphism of
  $\mathbb{A}_{\mathbb{R}}^{2}$ to that defined by the representation
  $\rho_{0}:\mathbb{S}^{1}\rightarrow\mathrm{SO}_{2,\mathbb{R}}$.
\end{prop}

\begin{proof}
  Indeed, let $(V=\mathbb{A}^2_{\mathbb{C}},\sigma_{\mathbb{A}^2_{\mathbb{R}}},\mu)$
  be a triple representing the given $\mathbb{S}^{1}$-action on
  $X=\mathbb{A}^2_{\mathbb{R}}$. Since by virtue of Lemma
  \ref{lem:Charac-S1-actions}, the $\GM$-action $\mu$ is hyperbolic,
it follows from Gutwirth's theorem \cite{Gu62} that $\mu$ is conjugate
  by an automorphism $\varphi$ of $\mathbb{A}_{\mathbb{C}}^{2}$ to a
  linear action $\nu$ of the form $t\cdot(u,v)=(t^{-p}u,t^{q}v)$ for some relatively prime
  positive integers $p$ and $q$. It follows that $X$ endowed with its $\mathbb{S}^1$-action is also 
  represented by the triple $(V,\sigma,\nu)$ where $\sigma=\varphi^*\sigma_{\mathbb{A}^2_{\mathbb{R}}}=\varphi^{-1} \sigma_{\mathbb{A}^2_{\mathbb{R}}}\varphi$ is the pull-back of $\sigma_{\mathbb{A}^2_{\mathbb{R}}}$ by $\varphi$. 
  
The algebraic quotient $V/\!/\GM$ is isomorphic to $\mathrm{Spec}(\mathbb{C}[z])\simeq \mathbb{A}^1_{\mathbb{C}}$, where $z=u^qv^p$. Letting $a$ and $b$ be positive integers such that $-ap+bq=1$, $s=u^av^b$ is a semi-invariant regular function of weight $1$ on $V$ which determines a $\GM$-equivariant isomorphism between $V$ and $V(\mathbb{A}^1_{\mathbb{C}},\DD')$ for the segmental divisor $\DD'=[-a/q,b/p]\otimes\{0\}$. Since $C=(V/\!/\GM,\overline{\sigma})$  is a real form of $\mathbb{A}^1_{\mathbb{R}}$, hence is isomorphic to the trivial one, there exists an automorphism $\psi:z\mapsto \alpha z +\beta $ of $V/\!/\GM$, where $\alpha\in \mathbb{C}^*$ and $\beta \in \mathbb{C}$, such that $\overline{\sigma}=\psi^*\sigma_{\mathbb{A}^1_{\mathbb{R}}}$. So $\overline{\sigma}$ is the composition of the automorphism $z\mapsto \overline{\alpha}\alpha^{-1}z+\alpha^{-1}(\overline{\beta}-\beta)$ of $V/\!/\GM$ with the complex conjugation. The condition $\overline{\sigma}^*\DD'=\flip{D'}+1\otimes \divi(h')$ for some $\overline{\sigma}$-invariant rational function $h' \in\mathbb{C}(z)$ then reads \[[-a/q,b/p]\otimes \{\alpha^{-1}(\overline{\beta}-\beta)/\}=[-b/p,a/q]\otimes \{0\}+1\otimes \divi(h').\] Since $\divi(h')$ is an integral Weil divisor, it follows that $(-ap+bq)/pq=1/pq$ is an integer. Thus $p=q=1$ and we can now assume further from the very beginning that $a=0$ and $b=1$, so that $s=v$ and $\DD'=[0,1]\otimes  \{ 0 \}$. The condition $\overline{\sigma}^*\DD'=\flip{D'}+1\otimes \divi(h')$ then implies that $\divi(h')=\{0\}=\{\alpha^{-1}(\overline{\beta}-\beta)\}$, hence that $\beta\in \mathbb{R}$ and $h'=\gamma z$ for some $\gamma$ in $\mathbb{C}^*$. The fact $h'$ is $\overline{\sigma}$-invariant implies in turn that $\gamma=c\alpha$ for some $c\in \mathbb{R}^*$. The phs-pair $(\DD',h')$ is thus the pull-back of the pair $(\DD'',h'')=([0,1]\otimes  \{ \beta \},cz)$ by the real isomorphism $\psi: (\mathbb{A}^1_{\mathbb{C}},\overline{\sigma})\stackrel{\simeq}{\rightarrow} (\mathbb{A}^1_{\mathbb{C}},\sigma_{\mathbb{A}^1_{\mathbb{R}}})$. Since $c,\beta \in \mathbb{R}$, we see that $(\DD'',h'')$ is in turn the pull-back of the phs-pair $(\DD,h)=([0,1]\otimes  \{ 0 \},z)$ by the real automorphism $\varphi: z\mapsto c(z-\beta)$ of $(\mathbb{A}^1_{\mathbb{C}},\sigma_{\mathbb{A}^1_{\mathbb{R}}})$. Summing up, we conclude that $X$ is  $\mathbb{S}^1$-equivariantly isomorphic to  $X(\mathrm{Spec}(\mathbb{R}[z]),([0,1]\otimes \{0\},z))$ hence to $\mathbb{A}_{\mathbb{R}}^{2}$ endowed with the action defined by representation $\rho_{0}$.
\end{proof}

\begin{example}[An algebraic model of the open Moebius band]
  Let $X'\simeq\mathbb{A}^1_{\mathbb{R}}\times \mathbb{S}^1$ be the real
  surface endowed with the free $\mathbb{S}^1$-action by translations
  on the second factor represented by the triple
  \[(V'=\mathbb{A}^1_{\mathbb{C}}\times
    \GM=\mathrm{Spec}(\mathbb{C}[w] [z^{\pm{1}}]),\sigma',\mu'),\]
  where $\sigma'$ is the composition of the involution
  $(w,z) \mapsto (w,z^{-1})$ with the complex conjugation and $\mu'$
  is the $\GM$-action by translations on the second factor.  A
  corresponding real DPD-presentation is given by the DPD-pair
  $(D,h)=(\{0\},1)$ on the real affine curve
  $C'=(\mathrm{Spec}(\mathbb{C}[w]),\tau)$ where $\tau$ is the complex
  conjugation.

  The involution $(w,z)\mapsto (-w,-z)$ of $V'$ is both $\sigma'$ and
  $\GM$-equivariant. The quotient of $V'$ by this involution thus
  inherits a real structure and a $\GM$-action which correspond to a
  smooth real affine surface $X$ with an effective
  $\mathbb{S}^1$-action. Explicitly, letting $x=wz^{-1}$ and $y=z^2$,
  $X$ is represented by the triple $(V,\sigma, \mu)$ where
  $V\simeq\mathbb{A}^1_{\mathbb{C}}\times
  \GM=\mathrm{Spec}(\mathbb{C}[x][y^{\pm{1}}])$, $\sigma$ is the
  composition of involution $(x,y)\mapsto (xy,y^{-1})$ with the
  complex conjugation, and $\mu$ is the effective $\GM$-action induced
  by the linear action with weights $(-1,2)$. A phs-pair for $X$
  is
  \[(\DD,h)=\left(\tfrac{1}{2}\otimes \{0\},x^2y\right) \mbox{ on the
      real affine curve
    }C=\left(\spec(\mathbb{C}[x^2y]),\tau\right)\simeq
    \mathbb{A}^1_{\mathbb{R}}\,,\] where $\tau$ denotes the complex
  conjugation. A real DPD-presentation of $X$ is given by the DPD-pair
  $(\DD(1),h)=\left(\tfrac{1}{2}\cdot \{0\},x^2y\right)$. 
  Furthermore, we have a commutative diagram of real
  morphisms:
  \[\xymatrix{ X'=(V',\sigma',\mu') \ar[rrrr]^{(w,z) \mapsto
        (x,y)=(wz^{-1},z^2)} \ar[d]_{\mathrm{p}_1} & & & &
      X=(V,\sigma,\mu) \ar[d]^{(x,y)\mapsto x^2y} \\ C'
      \ar[rrrr]^{w\mapsto x^2y=w^2} & & & & C.}\]

  The real locus $X'(\mathbb{R})$ of $X'$ endowed with its natural
  structure of differentiable manifold is diffeomorphic to
  $\mathbb{R}\times S^1$, on which $S^1$, identified with the set of
  complex numbers $\exp(i\theta)$ of norm $1$, acts by translations on
  the second factor. The real locus of $X$ is then diffeomorphic to
  the open Moebius band obtained as the quotient of
  $\mathbb{R}\times S^1$ by the involution
  $(w,\exp(i\theta))\mapsto (-w, \exp(-i\theta))$, endowed with the
  $S^1$-action induced by that on $\mathbb{R}\times S^1$.
\end{example}

\section{Higher dimensional examples}
\label{section4}

In this section, to continue to illustrate the methods to explicitly find
phs-pairs corresponding to given $\mathbb{S}^{1}$-actions, we present
two natural higher dimensional examples. We begin with a real form of
the action of the maximal torus $\mathbb{G}_{m,\mathbb{R}}$ of
$\mathrm{SL}_{2,\mathbb{R}}$ on $\mathrm{SL}_{2,\mathbb{R}}$ by
multiplication, whose algebraic quotient morphism turns out to provide
an algebraic model of the Hopf fibration $S^{3}\rightarrow S^{2}$.  We
then consider certain families of non-trivial forms of linear
$\mathbb{S}^{1}$-actions on $\mathbb{A}_{\mathbb{R}}^{4}$ constructed
by Moser-Jauslin \cite{MJ18}.

\subsection{An algebraic model of the Hopf fibration
  $S^{3}\rightarrow S^{2}$ }

Recall that the Hopf fibration $S^{3}\rightarrow S^{2}$ realizes the
real sphere $S^{3}$ as the total space of an $S^{1}$-torsor over the
real sphere $S^{2}$ in the category of differentiable real
manifolds. Namely, the circle $S^{1}$ identified with the set of
complex numbers $z=x+iy\in\mathbb{C}^{*}$ of norm $1$ acts by
component-wise multiplication on the real sphere
$S^{3}\subset\mathbb{R}^{4}=\mathbb{C}^{2}$ viewed as set of pairs of
complex numbers $(z_{1}=x_{1}+iy_{1},z_{2}=x_{2}+iy_{2})$ such that
$|z_{1}|^{2}+|z_{2}|^{2}=1$. Letting $S^{2}\subset\mathbb{R}^{3}$ be
the $2$-sphere with equation $x^{2}+y^{2}+z^{2}=1$, the quotient map
$S^{3}\longrightarrow S^{3}/S^{1}\simeq S^{2}$ is defined by
\[
  (z_{1},z_{2})=(x_{1},y_{1},x_{2},y_{2})\mapsto\left(x,y,z\right)=\left(\left|z_{1}\right|^{2}-\left|z_{2}\right|^{2},2\mathrm{Re}(z_{1}\overline{z}{}_{2}),2\mathrm{Im}(z_{1}\overline{z}_{2})\right)
\]
Putting
$\mathbb{S}^{3}=\mathrm{Spec}(\mathbb{R}[x_{1},y_{1},x_{2},y_{2}]/(x_{1}^{2}+y_{1}^{2}+x_{2}^{2}+y_{2}^{2}-1))$
an algebraic model of the action of $S^1$ on $S^3$ is given by the
restriction to $\mathbb{S}^{3}$ of the $\mathbb{S}^{1}$-action on
$\mathbb{A}_{\mathbb{R}}^{2}\times\mathbb{A}_{\mathbb{R}}^{2}=\mathrm{Spec}(\mathbb{R}[x_{1},y_{1}][x_{2},y_{2}])$
defined as the direct sum of two copies of the representation
$\rho_{0}:\mathbb{S}^{1}\rightarrow\mathrm{SO}_{2,\mathbb{R}}$ in
\eqref{eq:Representation}. An algebraic model of the quotient map is
given by the morphism of real algebraic varieties
\[
  p:\mathbb{S}^{3}\rightarrow\mathbb{S}^{2},\;(x_{1},y_{1},x_{2},y_{2})\mapsto(x_{1}^{2}+y_{1}^{2}-x_{2}^{2}-y_{2}^{2},2(x_{1}x_{2}+y_{1}y_{2}),2(x_{2}y_{1}-x_{1}y_{2})), \]
where
$\mathbb{S}^{2}=\mathrm{Spec}(\mathbb{R}[x,y,z]/(x^{2}+y^{2}+z^{2}-1))$. Recall
that the divisor class group of $\mathbb{S}^{2}$ is trivial whereas
the divisor class group of $\mathbb{S}^{2}_{\mathbb{C}}$ is isomorphic
to $\mathbb{Z}$, generated by the class of the Cartier divisor
$D=\{x+iy=1-z=0\}$.
\begin{prop}
  The real affine threefold $\mathbb{S}^{3}$ endowed with the
  so-defined $\mathbb{S}^{1}$-action is $\mathbb{S}^{1}$-equivariantly
  isomorphic to $X(\mathbb{S}^{2},(\{1\}\otimes D,1-z))$
\end{prop}

\begin{proof}
  As a real threefold with $\mathbb{S}^1$-action, $\mathbb{S}^{3}$ can
  be equivalently represented by the smooth complex affine quadric
  $V=\left\{ u_{1}v_{1}+u_{2}v_{2}=1\right\}$ in
  $\mathbb{A}^4_{\mathbb{C}}$, equipped with the real structure
  $\sigma$ defined as the composition of the involution
  $(u_{1},u_{2},v_{1},v_{2})\mapsto(v_{1},v_{2},u_{1},u_{2})$ with the
  complex conjugation, and endowed with the $\GM$-action $\mu$ defined
  by
  $t\cdot(u_{1},v_{1},u_{2},v_{2})=(tu_{1},t^{-1}v_{1},tu_{2},t^{-1}v_{2})$. The
  algebraic quotient $V/\!/\GM$ is isomorphic to the smooth affine
  quadric $S=\left\{ uv+z^{2}=1\right\}$ in
  $\mathbb{A}_{\mathbb{C}}^{3}$ and the quotient morphism
\[V\rightarrow
  V/\!/\mathbb{G}_{m,\mathbb{C}}=S,\quad(u_{1},v_{1},u_{2},v_{2})\mapsto(u,v,w)=(2u_{1}v_{2},2u_{2}v_{1},2u_{1}v_{1}-1) \]
is a $\GM$-torsor whose class in
$H^{1}(S,\mathcal{O}_{S}^{*})\simeq\mathrm{Pic}(S)\simeq\mathbb{Z}$
coincides with that of the line bundle associated to the Cartier
divisor $D'=\{u=1-z=0\}$ on $S$. It follows that $S$ is the
AH-quotient of $V$ and that $\mathcal{D}'=\{1\}\otimes D'$ is a proper
segmental divisor on $S$ such that $V$ is equivariantly isomorphic to
$V(S,\mathcal{D'})$.

The real structure $\sigma$ descends on $S$ to the real structure
$\overline{\sigma}$ defined as the composition of the involution
$(u,v,w)\mapsto(v,u,w)$ with the complex conjugation. Since
$\overline{\sigma}^*D'=\{v=1-z=0\}=-D'+\divi(1-z)$ where $1-z$ is
$\overline{\sigma}$-invariant, we conclude that $\mathbb{S}^{3}$ is
equivariantly isomorphic to
$X((S,\overline{\sigma}),(\{1\}\otimes D',1-z))$. The assertion then
follows by noticing that the phs-pair $(\{1\}\otimes D, 1-z)$ is the
pull-back of $(\{1\}\otimes D',1-z)$ by the isomorphism of real
algebraic
surfaces
\[\varphi:(\mathbb{S}_{\mathbb{C}}^{2},\sigma_{\mathbb{S}^{2}})\stackrel{\simeq}{\longrightarrow}
  (S,\overline{\sigma}), \quad (x,y,z)\mapsto (u,v,z)=(x+iy,x-iy,z).\]
\end{proof}

More generally, recall that for every integer $p\geq 1$, the Lens
space $L(p,1)$ is the quotient of
$S^3=\{(z_1,z_2),\, |z_{1}|^{2}+|z_{2}|^{2}=1\}$ by the free action of
the group $\mathbb{Z}_p$ defined by
$(z_1,z_2)\mapsto (\zeta z_1,z_2)$, where $\zeta=\exp(2i\pi/p)$. This
action is equivariant with respect to the $S^1$-action on $S^3$ by
component-wise multiplication, and so, $L(p,1)$ inherits a effective
action of $S^1$. A similar argument as in the proof of the previous
proposition shows that an algebraic model of $L(p,1)$ endowed with
this $S^1$-action is given by the real affine threefold
$\mathbb{L}(p,1):=X(\mathbb{S}^{2},(\{1\}\otimes (pD),(1-z)^p))$.

\subsection{Linear and non-linearizable $\mathbb{S}^{1}$-actions on
  $\mathbb{A}_{\mathbb{R}}^{4}$ }

Let again $\rho_{0}:\mathbb{S}^{1}\rightarrow SO_{2,\mathbb{R}}$ be
the representation defined in \eqref{eq:Representation}.  For every
integer $r\geq 0$, the morphism
\[\nu_{2,r}:\mathbb{S}^{1}\times\mathbb{A}_{\mathbb{R}}^{4}\rightarrow\mathbb{A}_{\mathbb{R}}^{4},
  \quad
  (s,(u_1,v_1,u_2,v_2))\mapsto(\rho_{0}(s)^{2}\cdot(u_1,v_1),\rho_{0}(s)^{2r+1}\cdot(u_2,v_2))\]
defines an effective action of $\mathbb{S}^{1}$ on
$\mathbb{A}_{\mathbb{R}}^{4}=\mathrm{Spec}(\mathbb{R}[u_1,v_1][u_2,v_2])$.
With the notation of Lemma \ref{lem:S1-action-equiv-form}, the latter
is represented by the triple $(V,\sigma,\mu_{2,r})$, where
$V=\mathbb{A}_{\mathbb{C}}^{4}=\mathrm{Spec}(\mathbb{C}[a,b][x,y])$,
$\sigma$ is the real structure defined as the composition of the
involution $(a,b,x,y)\mapsto (b,a,y,x)$ with the complex conjugation,
and $\mu_{2,r}$ is the linear hyperbolic $\GM$-action weights
$(2,-2,2r+1,-2r-1)$.

For $r=1$, Freudenburg and Moser-Jauslin \cite{FMJ04} constructed an
$\mathbb{S}^1$-action $\nu'_{2,1}$ on $\mathbb{A}_{\mathbb{R}}^{4}$
which is a non-trivial form of $\nu_{2,1}$, hence in particular a
non-linearizable action. The construction was generalized later on by
Moser-Jauslin \cite{MJ18} for arbitrary $r\geq 2$ to yield infinite
families of pairwise non-conjugate non-linearizable
$\mathbb{S}^{1}$-actions on $\mathbb{A}_{\mathbb{R}}^{4}$. Our aim is
to give a complementary description of these actions in terms of
phs-pairs.

Given $r\geq 1$, we let $Q_{2,r}$ be the closed subvariety of
$\mathbb{A}_{\mathbb{C}}^{4}=\mathrm{Spec}(\mathbb{C}[u,v,z,w])$ with
equation $uv=z^{2r+1}w^2$ and we let $\pi:Y_{2,r}\rightarrow Q_{2,r}$
be the blow-up of $Q_{2,r}$ with center at the closed subscheme $W$
with defining ideal $(u,v,z^{2r+1},w^{2})$. We denote by $E$ be the
exceptional divisor of $\pi$, and by $D_{z,u}$, $D_{z,v}$, $D_{w,u}$
and $D_{w,v}$ the respective proper transforms in $Y_{2,r}$ of the
Weil divisors $\{z=u=0\}$, $\{z=v=0\}$, $\{w=u=0\}$ and $\{w=v=0\}$ on
$Q_{2,r}$. We let $\DD_{2,r}$ be the segmental divisor on $Y_{2,r}$
defined by
\[\mathcal{D}_{2,r}=r\otimes D_{z,v}+\{1\}\otimes D_{w,v}+[2r,2r+1]\otimes E.\]
The main result of \cite{MJ18} can now be reformulated as follows:

\begin{thm} Let $r\geq 1$ be a fixed integer. Then for every
  polynomial $P\in\mathbb{R}[z]$, there exists a real structure
  $\tau_P$ on $Y_{2,r}$ with the following properties:
  \begin{enumerate}[(1)]
  \item The rational function $h_P= z^{r}((1-zP^{2}(z))w+P^{n}(z)v)$
    on $Y_{2,r}$ is $\tau_P$-invariant,
  \item $(\mathcal{D}_{2,r},h_P)$ is phs-pair on
    $Z_{2,r,P}=(Y_{2,r},\tau_P)$,
  \item The affine fourfold $X(Z_{2,r,P},(\mathcal{D}_{2,r},h_P))$ is
    isomorphic to $\mathbb{A}^4_{\mathbb{R}}$.
  \end{enumerate}
  Furthermore, two fourfolds
  $X(Z_{2,r,{P_i}},(\mathcal{D}_{2,r},h_{P_i}))$, $i=1,2$, are
  $\mathbb{S}^1$-equivariantly isomorphic if and only if there exists
  $c\in \mathbb{R}^*$ such that $P_2(z)\equiv cP_1(c^2z)$ modulo
  $z^r$.
\end{thm}

\begin{proof}
Letting $n=2r+1$, the matrix 
\[
M_{P}=\left(\begin{array}{cc}
1-abP^{2}(ab) & a^{n}P^{n}(ab)\\
-b^{n}P^{n}(ab) & \sum_{j=0}^{2r}(abP^{2}(ab))^{j}
\end{array}\right)\in\mathrm{SL}_{2}(\mathbb{C}[a,b])
\]
defines an automorphism $\varphi_{P}$ of $V$ over
$\mathrm{Spec}(\mathbb{C}[a,b])$ which is equivariant for the
$\GM$-action $\mu_{2,r}$. By \cite[Theorem 3.1 (i)]{MJ18}, the
composition $\sigma_{P}=\varphi_{P} \circ \sigma$ is real structure on
$V$ such that $(V,\sigma_P)$ is isomorphic to
$\mathbb{A}^4_{\mathbb{R}}$. Since $\varphi_{P}$ is $\GM$-equivariant,
it follows from Lemma \ref{lem:S1-action-equiv-form} that the triple
$(V,\sigma_P,\mu_{2,r})$ is a smooth real affine fourfold $X_P$
endowed with an effective $\mathbb{S}^1$-action which is a real form
of $X_0=(V,\sigma,\mu_{2,r})$. The main result in \cite{MJ18} asserts
that $X_{P_1}$ is $\mathbb{S}^1$-equivariantly isomorphic to $X_{P_2}$
if and only if there exists $c\in\mathbb{R}^{*}$ such that
$P_{2}(z)\equiv cP_{1}(c^{2}z)$ modulo $z^{r}$. So to complete the
proof, it suffices to show that the AH-quotient of $V$ is equal to
$Y_{2,r}$, that $V$ is $\GM$-equivariantly isomorphic to
$V(Y_{2,r},\DD_{2,r})$, and that $X_P$ is $\mathbb{S}^1$-equivariantly
isomorphic to $X(Z_{2,r,P},(\DD_{2,r},h_P))$ for the claimed rational
function $h_P$ on $Y_{2,r}$ endowed with the real structure $\tau_P$
induced by $\sigma_P$.

First, it is clear that $Q_{2,r}$ is the algebraic quotient of $V$,
the quotient morphism being given by
$(a,b,x,y)\mapsto(u,v,z,w)=(a^{n}y^{2},b^{n}x^{2},ab,xy)$. The fact
that $Y_{2,r}$ is the AH-quotient of $V$ then follows from the
observation that the minimal $d$ for which the graded sub-algebra of
$A=\mathbb{C}[a,b,x,y]$ consisting of semi-invariants of positive
(resp. negative) weights divisible by $d$ is generated in degree $1$,
is equal to $2n$, and that
$A_{2n}\cdot
A_{-2n}=\left(a^{n}y^{2},b^{n}x^{2},a^{n}b^{n},x^{2}y^{2}\right)=\left(u,v,z^{n},w^{2}\right)$
is precisely the defining ideal of the center of the blow-up
$\pi:Y_{2,r}\rightarrow Q_{2,r}$.

The fact that $V\simeq V(Y_{2,r},\DD_{2,r})$ can then be derived for
instance from the toric downgrading method described in \cite[Section
11]{AH}. In our case, $V$ endowed with the $\GM$-action $\mu_{2,r}$ is
the restriction to the sub-torus $\GM\hookrightarrow\GM^{4}$,
$t\mapsto(t^{2},t^{-2},t^{n},t^{-n})$ of the usual structure of toric
variety of $\mathbb{A}_{\mathbb{C}}^{4}$. This sub-torus corresponds
to the injection $F:\mathbb{Z}\rightarrow\mathbb{Z}^{4}$,
$1\mapsto(2,-2,n,-n)$ between the respective lattices of $1$-parameter
subgroups, and we have an exact sequence
$0\rightarrow\mathbb{Z}\stackrel{F}{\rightarrow}\mathbb{Z}^{4}\stackrel{G}{\rightarrow}\mathbb{Z}^{3}\rightarrow0$
where
\[
G:\mathbb{Z}^{4}\rightarrow\mathbb{Z}^{3} \quad \mbox{is given by the
  matrix} \quad \left(\begin{array}{cccc}
1 & 1 & 0 & 0\\
0 & 0 & 1 & 1\\
n & 0 & 0 & 2
\end{array}\right).
\]

Let $\Sigma'$ be the fan in $\mathbb{Z}^{4}\otimes_\ZZ\RR$ generated
by the cone $\mathrm{cone}(e_{1},e_{2},e_{3},e_{4})$ where $e_{i}$ are
the standard basis vectors. The coarsest fan $\Sigma$ in
$\mathbb{Z}^{3}\otimes_\ZZ\RR$ generated by the image by $P$ of
$\Sigma'$ is the simplicial fan generated by $f_{1}=(1,0,n)$,
$f_{2}=(1,0,0)$, $f_{3}=(0,1,0)$, $f_{4}=(0,1,2)$ and the additional
vector $f_{5}=(2,n,2n)=2f_{1}+nf_{3}=2f_{2}+nf_{4}$.  This fan
describes $Y_{2,r}$ as a toric threefold in which the invariant
divisors corresponding the rays generated by the $f_{i}$ are
respectively $D(f_{1})=D_{z,u}$, $D(f_{2})=D_{z,v}$,
$D(f_{3})=D_{w,v}$, $D(f_{4})=D_{w,u}$, while $D(f_{5})$ is the
exceptional divisor $E$ of $\pi:Y_{2,r}\rightarrow Q_{2,r}$. A direct
calculation now shows that $\DD_{2,r}$ is equal to proper segmental
divisor supported on the union of the $D(f_i)$ whose coefficients are
the images by the section
$\gamma=r\mathrm{pr}_{2}+\mathrm{pr}_{3}:\mathbb{Z}^{4}\rightarrow\mathbb{Z}$
of $F$ of the segments
$\gamma\left(\mathbb{R}_{\geq0}^{4}\cap P^{-1}(f_{i})\right)$,
$i=1,\ldots ,5$. This implies that $V\simeq V(Y_{2,r},\DD_{2,r})$

Now recall from the proof of Theorem \ref{thm:real-AH-Main} and Remark
\ref{rem:semi-inv-w1} that $X_P$ is then $\mathbb{S}^1$-equivariantly
isomorphic to $X(Z_{2,r,P},(\mathcal{D}_{2,r},s\sigma_P^*s))$, where
$s$ is the semi-invariant rational function of weight $1$ on
$\mathbb{A}_{\mathbb{C}}^{4}$ which provides the identification
$\mathbb{C}[a,b,x,y]_m=\Gamma(Y_{2,r},\mathcal{O}_{Y_{2,r}}(\mathcal{D}_{2,r}(m)))\cdot
s^m$ for every $m\in \mathbb{Z}$. Our choice of section $\gamma$ of
$F$ in the toric downgrading method used to construct $\DD_{2,r}$
corresponds to the choice $s=b^rx$, and the fact that
$h_P=(b^rx)\sigma_P^*(b^rx)$ then follows from a direct calculation.
\end{proof}

\bibliographystyle{amsalpha}

\end{document}